\pdfoutput=1

\documentclass[10pt]{article}

\usepackage{preamble}
\usepackage{courier}
\usepackage[square,numbers]{natbib}

\title{\textbf{Some model theory of }$\Th(\mathbb{N},\cdot)$}
\author{Atticus Stonestrom\footnote{Email: \texttt{atticusstonestrom@yahoo.com}}}
\date{}

\begin{document}
\maketitle

\begin{abstract}
`Skolem arithmetic' is the complete theory $T$ of the multiplicative monoid $(\mathbb{N},\cdot)$. We give a full characterization of the $\varnothing$-definable stably embedded sets of $T$, showing in particular that, up to the relation of having the same definable closure, there is only one non-trivial one: the set of squarefree elements. We then prove that $T$ has weak elimination of imaginaries but not elimination of finite imaginaries.
\end{abstract}

\renewcommand{\abstractname}{Acknowledgements}
\begin{abstract}
 I would first like to thank Jamshid Derakhshan for his constant support and friendship, and for all of the math he has taught me over the years. I would also like to thank Alexi Block Gorman and James Hanson for their insights on Presburger arithmetic, Udi Hrushovski for a helpful discussion of the paper, and Emil Je{\v r}{\'a}bek, for his succinct and elegant post on MathOverflow summarizing a quantifier elimination result for $T$.

Finally, I would especially like to thank Alex Kruckman; his guidance and patience, not only over the course of this paper, but also during all of these past months as I have been learning about (neo-)stability theory, are simply impossible to overstate. I would like to thank him in particular for his great help in understanding the QE result for $T$, for pointing out an oversight in the proof of Lemma 3.2, for his observation that Fact B.2 follows from a theorem of Tarski, and for all of his feedback and encouragement.
\end{abstract}

\newpage

\section{Notation}
Throughout, $T$ denotes the complete theory of the multiplicative monoid $(\mathbb{N},\cdot,1)$, where we use $\mathbb{N}$ to mean the non-zero natural numbers. Fix $\mathbb{M}$ to be a monster model of $T$: concretely, we throughout consider $\mathbb{M}$ as a special structure \citep{baldwin_2015} of size $\beth_{\mu^+}(\mu)$, for some suitably large cardinal $\mu$, and we use the word `small' to refer to subsets of size $<\mu^+$. In particular, whenever we refer to a `parameter set $B\subset\mathbb{M}$', we mean that $B$ is small. Now, define a formula $v_1\mid v_2\equiv\exists w (v_2=v_1\cdot w)$, meaning `$v_1$ divides $v_2$', and define further a formula $$\pi(u)\equiv (u\neq 1)\wedge\forall w_1\forall w_2\left[u\mid (w_1\cdot w_2)\to(u\mid w_1\vee u\mid w_2)\right],$$ meaning `$u$ is prime'. We will refer to any realization of $\pi$ as a `prime', and we will denote by $\mathbb{P}$ the set $\pi(\mathbb{M})$ of all primes of $\mathbb{M}$. All other formulas and notation specific to $\mathbb{M}$ will be defined where they appear.

We use $\omega$ to refer to the natural numbers \textit{including} $0$, and restrict the symbol $\mathbb{N}$ to mean only the substructure $\N\prec\mathbb{M}$; we instead use $\omega$ when the natural numbers are employed as an index set. For $n\in\omega$, we use $[n]$ to denote the set $\{1,\dots,n\}$. We use $\subseteq$ for containment and $\subset$ for strict containment, and $|\overline{a}|$ for the length of a tuple $\overline{a}$. Given an equivalence relation $E\subseteq\mathbb{M}^n\times\mathbb{M}^n$, we write the $E$-equivalence class of a tuple $\overline{a}\in\mathbb{M}^n$ as $\overline{a}^E$. 

We often use `partitioned formulas' of the form $\phi(\overline{v};\overline{w})$, where $\overline{w}$ represents a parameter tuple. Given a tuple $\overline{c}\in\mathbb{M}^{|\overline{w}|}$ and a subset $X\subseteq\mathbb{M}$, we write $\phi(X^{|\overline{v}|};\overline{c})$ to mean the set of all tuples $\overline{a}\in X^{|\overline{v}|}$ such that $\mathbb{M}\models\phi(\overline{a};\overline{c})$. Finally, given tuples $\overline{a}_1$ and $\overline{a}_2$ and a parameter set $B\subset\mathbb{M}$, we write $\overline{a}_1\equiv_B\overline{a}_2$ to mean that $\tp(\overline{a}_1/B)=\tp(\overline{a}_2/B)$, and use $\overline{a}_1\equiv\overline{a}_2$ to abbreviate $\overline{a}_1\equiv_\varnothing\overline{a}_2$.


\section{Introduction}
Skolem arithmetic is in many regards a highly pathological theory, lying firmly in the `delta quadrant' of the model-theoretic universe. For instance, the formula $v\mid w$ has both the independence property, as it shatters any small subset of $\mathbb{P}$, and the strict order property, as it induces a partial ordering on $\mathbb{M}$ without maximal elements. Similarly, model-theoretic notions of independence behave poorly in the theory; forking and dividing do not coincide, and in fact types may fork over their domains.\footnote{As an example, if $r\in\mathbb{M}$ is any squarefree element with infinitely many prime factors, and $p\in\mathbb{P}$ is one of its prime factors, then $p\nind^f_{\{r\}}r$.} Our aim in this paper is to instead understand what kinds of \textit{tame} structure appear within $T$. To this end, we compute the $\varnothing$-definable stably embedded sets of $T$, and we show that $T$ admits weak elimination of imaginaries.

The paper proceeds as follows. In Section 2 we give a number of constructions and lemmas that provide the main tools for the rest of the paper. First, in Section 2.1, we describe a quantifier elimination result for $T$ noted in unpublished work of Emil Je{\v r}{\'a}bek; for completeness, we include a proof of this result in Appendix A. In Section 2.2 we discuss radical, or squarefree, elements; there is a natural two-to-one correspondence between the $\mathbb{M}$-definable subsets of $\mathbb{P}$ and the radical elements of $\mathbb{M}$, which allows us to employ radical elements as a means of coding definable sets of primes. In Section 2.3 we give a few lemmas on conjugate tuples, and in Section 2.4 we prove some results on definability; we discuss in particular a lexicographic ordering on $\mathbb{M}^n$ that allows us to obtain some control over the definable closure of a tuple. Finally, in Section 2.5, we discuss $\varnothing$-definable subsets of $\mathbb{M}$ that are cofinal with respect to divisibility, and prove that the definable closure of any such set is all of $\mathbb{M}$.

In Section 3 we investigate the theory's $\varnothing$-definable stably embedded sets, a full characterization of which we give in Theorem 3.4. In short, a $\varnothing$-definable subset $D\subseteq\mathbb{M}^n$ is stably embedded if and only if $\dcl D$ is one of $\{1\}$, $\mathbb{M}$, or the definable closure of the set of radical elements of $\mathbb{M}$. Thus, up to the relation of having the same definable closure, the only non-trivial stably embedded subset of $\mathbb{M}$ is the set of radical elements. Note that, by the results on cofinal sets in Section 2, every $\varnothing$-definable cofinal subset of $\mathbb{M}$ is stably embedded.

In Section 4 we investigate the theory's imaginary elements. We first give an easy example showing that the theory does not eliminate finite imaginaries; although any $\mathbb{M}$-definable subset of $\mathbb{P}$ \textit{can} be coded (by a suitable radical element), the next simplest case cannot be in general, and the set $\{p\cdot q,r\cdot s\}$ has no canonical parameter for any distinct primes $p,q,r,s$. With this done, the bulk of the section is then devoted to showing that $T$ nonetheless does have weak elimination of imaginaries, which we prove in Theorem 4.9. The approach is to reduce questions about sets definable over arbitrary parameters to questions about sets definable over radical tuples, which are generally more tractable; we do this using both results from Section 2 and the stable embeddedness of the set of radical elements.


\section{Lemmas and constructions}
\subsection{Quantifier elimination}

Define a formula $\operatorname{pow}_u(w)\equiv\pi(u)\wedge\forall  v\left[\pi(v)\wedge (v\mid w)\to (v=u)\right]$, expressing that $u$ is prime and the only possible prime divisor of $w$. We have the following theorem of $\mathbb{N}$: $$\mathbb{N}\models\forall u\forall v\big[\pi(u)\to\exists w\bigl[\pow_u(w)\wedge (w\mid v)\wedge\forall z\bigl(\pow_u(z)\wedge(z\mid v)\to (z\mid w)\bigl)\big]\big].$$ In other words, for every prime $u\in\mathbb{P}$ and every $v\in\mathbb{M}$, there exists an element $w$ that is the `highest power' of $u$ dividing $v$. Such a $w$ is unique with this property, and we will denote it by $v_{u^\infty}$, which we call the `$u$-part' of $v$. We also will denote by $u^\infty:=\pow_u(\mathbb{M})$ the set of all powers of $u$ in $\mathbb{M}$.

 Fix a prime $p\in\mathbb{P}$ and a formula $\phi(\overline{v})$. Define the `relativization to $p$' of $\phi$, denoted $\phi^p(\overline{v})$, by induction. If $\phi$ is atomic, take $\phi^p(\overline{v})$ to be the formula obtained by replacing every instance of $v_i$ in $\phi$ with $v_{i,p^\infty}$. Then let relativization to $p$ distribute over Boolean operations and existential quantification; for example, if $\phi(\overline{v})$ is of the form $\exists w\psi(\overline{v},w)$, then take $\phi^p(\overline{v})\equiv\exists w\psi^p(\overline{v},w)$. 

\begin{remark}
The `relativization to $p$' process simply replaces every (free or bound) instance of a variable in $\phi(\overline{v})$ with its $p$-part. In other words, $\phi^p(\overline{v})$ just expresses that $\phi(\overline{v}_{p^\infty})$ holds in the submonoid $p^\infty$ of $\mathbb{M}$. Note in particular that $\phi^p(\overline{v})$ is equivalent to $\phi^p(\overline{v}_{p^\infty})$.
\end{remark}

Now, given a formula $\phi(\overline{v})$ and a natural number $n\in\omega$, we can construct a corresponding formula $\exists^{\geqslant n}u\left[\pi(u)\wedge\phi^u(\overline{v})\right]$, which expresses that there are at least $n$ primes $u$ at which the $u$-relativization of $\phi(\overline{v})$ holds. In fact, every formula is equivalent modulo $T$ to a Boolean combination of formulas of this form; this is noted in unpublished work of Emil Je{\v r}{\'a}bek, who attributes it to a more general result of Mostowski. \citep{jerabek_2018} For completeness, we include a self-contained proof of this fact in Appendix A. To ease the notational strain, we will throughout the paper write these formulas and their negations as $\phi^{\geqslant n}(\overline{v})$ and $\phi^{< n}(\overline{v})$, respectively, and use $\phi^{=n}(\overline{v})$ to mean $\phi^{\geqslant n}(\overline{v})\wedge\phi^{\leqslant n}(\overline{v})$. We record here only the statement of the main QE result; for a proof see Appendix A.

\begin{manualtheorem}{A.4}
Any formula $\phi(\overline{v})$ is equivalent to a Boolean combination of formulas of the form $\theta^{\geqslant m}(\overline{v})$.$\hfill\blacksquare$
\end{manualtheorem}

\subsection{Radical elements}
For a tuple $\overline{a}$, we write $\sigma(\overline{a})$ to mean $\left\{p\in\mathbb{P}:\overline{a}_{p^\infty}\neq \overline{1}\right\}$, which we call the `support' of $\overline{a}$. Given two tuples $\overline{a}_1$ and $\overline{a}_2$, we write $\overline{a}_1\perp\overline{a}_2$ to mean that $\sigma(\overline{a}_1)\cap\sigma(\overline{a}_2)=\varnothing$; note that both $\sigma(\overline{v})$ for a variable tuple $\overline{v}$ and $\overline{v}_1\perp\overline{v}_2$ for variable tuples $\overline{v}_1$ and $\overline{v}_2$ are first-order expressible.

We will call an element $r\in\mathbb{M}$ squarefree, or radical, if $r_{p^\infty}\in\{1,p\}$, ie if $p^2\nmid r$, for every prime $p\in\mathbb{P}$. We have the following formula defining radical elements, $$\rad(v)\equiv\forall u\left[\pi(u)\to (v_{u^\infty}=1)\vee (v_{u^\infty}=u)\right],$$ and the following theorem of $T$: $\mathbb{N}\models\forall v\exists !w\left[\rad(w)\wedge \sigma(v)=\sigma(w)\right]$. In other words, for every $a\in\mathbb{M}$, there exists a unique radical element $r\in\mathbb{M}$ with $\sigma(r)=\sigma(a)$. (Indeed, for an arbitrary natural number $a\in\mathbb{N}$, the squarefree part of $a$ witnesses this sentence in $\mathbb{N}$ when $v=a$, which gives the desired result.) We will denote the element $r$ by $\sqrt{a}$, and we will write $\sqrt{\mathbb{M}}=\rad(\mathbb{M})$ to mean the set of all radical elements of $\mathbb{M}$. If one thinks of elements of $\mathbb{M}$ as coding families of powers of primes, then $\sqrt{a}$ can be thought of as containing \textit{precisely} the information of which primes are represented in $a$, and nothing more. Note that a radical element is completely determined by its prime support. Conversely, any $\mathbb{M}$-definable set of primes can be coded by a suitable radical element, in the following way:

\begin{lemma}
For any $\mathbb{M}$-definable set of primes $Q\subseteq\mathbb{P}$, there exists a radical element $r\in\mathbb{M}$ such that $\sigma(r)$ is one of $Q$ and $\mathbb{P}\setminus Q$.
\end{lemma}
\begin{proof}
Use $\neg_0$ denote the string `$\neg$' and $\neg_1$ to denote the empty string. For a fixed formula $\theta(u;\overline{w})$, we show the following:
 $$\mathbb{N}\models\forall \overline{w}\exists v\left[\rad(v)\wedge\bigvee_{i=0}^1\forall u\left[\pi(u)\to \big((u\mid v)\leftrightarrow \neg_i\theta(u;\overline{w})\big)\right]\right].$$ To see this, let $\overline{c}\in\mathbb{N}^{|\overline{w}|}$ be arbitrary, and consider the set $P:=\left\{p\in\mathbb{P}\cap\mathbb{N}:\mathbb{N}\models\theta(p,\overline{c})\right\}$. We have two cases; first suppose that $P\subseteq\sigma(\overline{c})$. Then in particular, since elements of $\mathbb{N}$ have finite support, $P$ is finite, so taking $v=\prod P\in\mathbb{N}$ and $i=1$ gives the desired result. Otherwise, suppose there is a prime $p\in P\setminus\sigma(\overline{c})$. Then, for every $q\in\mathbb{P}\cap\mathbb{N}$ not in $\sigma(\overline{c})$, there is an automorphism of $\mathbb{N}$ taking $p$ to $q$ and fixing the (prime factors of) $\overline{c}$, whence $q\in P$. So the complement of $P$ in $\mathbb{P}\cap\mathbb{N}$ is a subset of $\sigma(\overline{c})$, and is therefore again finite, so taking $v=\prod\left[(\mathbb{P}\cap\mathbb{N})\setminus P\right]\in\mathbb{N}$ and $i=0$ gives the desired result.
\end{proof}

There is thus a natural two-to-one correspondence from the $\mathbb{M}$-definable sets of primes onto $\sqrt{\mathbb{M}}$. We can avoid worrying about complements if we restrict our attention to the $\mathbb{M}$-definable subsets of $\sigma(\overline{a})$ for some tuple $\overline{a}$; we will use the notation introduced in the following remark essentially without mention throughout the rest of the paper.

\begin{remark}
If $Q$ is any $\mathbb{M}$-definable set of primes, and $\overline{a}$ is any tuple, then there exists a tuple $\overline{a}_Q$ such that, for any $p\in\mathbb{P}$, we have $\overline{a}_{Q,p^\infty}=\overline{a}_{p^\infty}$ if $p\in Q$ and $\overline{a}_{Q,p^\infty}=\overline{1}$ if $p\notin Q$. We abbreviate $\overline{a}_{\mathbb{P}\setminus Q}$ by writing $\overline{a}_{\neg Q}$, and we abuse notation by writing $$\overline{a}_Q=\prod\left\{\overline{a}_{q^\infty}:q\in Q\right\}\in\mathbb{M}^{|\overline{a}|}.$$ As a particular case, if there exists a tuple $\overline{a}$ such that $Q\subseteq\sigma(\overline{a})$, then there exists a radical element $r\in\sqrt{\mathbb{M}}$ with $\sigma(r)=Q$; we similarly abuse notation by writing $r=\prod Q$.
\end{remark}
\subsection{Conjugates}
We collect here a few easy criteria for determining and constructing conjugate tuples.

\begin{lemma} Suppose $\overline{a}_1,\overline{a}_2$ and $\overline{c}_1,\overline{c}_2$ are tuples, and that $\overline{a}_1\perp\overline{c}_1$ and $\overline{a}_2\perp\overline{c}_2$. If $\overline{a}_1\equiv\overline{a}_2$ and $\overline{c}_1\equiv\overline{c}_2$, then $\overline{a}_1\overline{c}_1\equiv\overline{a}_2\overline{c}_2$.
\end{lemma}
\begin{proof}
Let $\phi(\overline{v},\overline{w})$ be any formula; by Theorem A.4, we may assume that $\phi$ is of form $\theta^{\geqslant n}(\overline{v},\overline{w})$ for some formula $\theta$ and $n\in\omega$. Suppose that $\theta^{\geqslant n}(\overline{a}_1,\overline{c}_1)$ holds; we will show that $\theta^{\geqslant n}(\overline{a}_2,\overline{c}_2)$ does too. If $\mathbb{M}\models\theta^p(\overline{1},\overline{1})$ for some (and hence any) $p\in\mathbb{P}$, then the result is immediate, so assume this is not the case, and let $p_1,\dots,p_n$ be any distinct primes such that $\mathbb{M}\models\theta^{p_i}(\overline{a}_1,\overline{c}_1)$ for each $i\in[n]$. Since $\mathbb{M}\models\neg\theta^p(\overline{1},\overline{1})$ and $\overline{a}_1\perp\overline{c}_1$, for each $i\in[n]$ precisely one of $p_i\in\sigma(\overline{a}_1)$ and $p_i\in\sigma(\overline{c}_1)$ holds. Assume without loss of generality that $n'\in[n+1]$ is such that, for any $i\in[n]$, we have $p_i\in\sigma(\overline{a}_1)$ if $i<n'$ and $p_i\in\sigma(\overline{c}_1)$ if $i\geqslant n'$. Also, by the theorem hypotheses, let $f,g$ be any automorphisms such that $f(\overline{a}_1)=\overline{a}_2$ and $g(\overline{c}_1)=\overline{c}_2$. Since $\overline{a}_2$ and $\overline{c}_2$ are coprime, we have $\{f(p_1),\dots,f(p_{n'-1})\}\perp\{g(p_{n'}),\dots,g(p_{n})\}$, and it thus suffices to show $\mathbb{M}\models\theta^{f(p_i)}(\overline{a}_2,\overline{c}_2)$ and $\mathbb{M}\models\theta^{g(p_j)}(\overline{a}_2,\overline{c}_2)$ for each $i<n'$ and $j\geqslant n'$. We show the former, as the latter case is identical; fix $i<n'$. Since $\overline{a}_1\perp\overline{c}_1$ and $p_i\in\sigma(\overline{a}_1)$, by Remark 2.1 we have $\mathbb{M}\models\theta^{p_i}(\overline{a}_1,\overline{1})$; pushing forward along $f$ tells us $\mathbb{M}\models\theta^{f(p_i)}(\overline{a}_2,\overline{1})$. But $\overline{a}_2\perp\overline{c}_2$, and $f(p_i)\in\sigma(\overline{a}_2)$, so again by Remark 2.1 we have $\mathbb{M}\models\theta^{f(p_i)}(\overline{a}_2,\overline{c}_2)$, as desired.
\end{proof}

\begin{lemma}
Suppose $r_1,r_2\in\sqrt{\mathbb{M}}$. If $\sigma(r_1)$ and $\sigma(r_2)$ are infinite, then $r_1\equiv r_2$.
\end{lemma}
\begin{proof}
This is essentially an immediate corollary of Theorem A.4; we need only show that $\theta^{\geqslant n}(r_1)$ holds if and only if $\theta^{\geqslant n}(r_2)$ holds for any formula $\theta$ and $n\in\omega$. But the sentence $\theta^p(1)$ (respectively $\theta^p(p)$) holds for some $p\in\mathbb{P}$ if and only if $\theta^q(1)$ (respectively $\theta^q(q)$) holds for every $q\in\mathbb{P}$, whence the result follows.
\end{proof}

\begin{lemma}
If $\overline{a}$ is any tuple and $B\subset\mathbb{M}$ is any (small) parameter set, then there exists $\overline{a}'\in\mathbb{M}^{|\overline{a}|}$ a realization of $\tp(\overline{a})$ such that $\overline{a}'\perp B$.
\end{lemma}
\begin{proof}
By compactness, and since $B$ is small, it suffices to replace $\tp(\overline{a})$ with a single formula $\phi(\overline{v})\in\tp(\overline{a})$ and to assume that $B$ is finite, say of form $\{b_1,\dots,b_n\}$. Now the desired result follows from the following theorem of $\mathbb{N}$: $$\mathbb{N}\models\forall w_1\dots\forall w_n\exists \overline{v}\left[\phi(\overline{v})\wedge\bigwedge_{i\in[n]}\overline{v}\perp w_i\right].$$ To see this, let $\overline{e}$ be any realization of $\phi(\overline{v})$ in $\mathbb{N}$ and $\overline{c}$ any instantiation of $\overline{w}$ in $\mathbb{N}$. Choose any set of primes $P\subset\mathbb{P}\cap\mathbb{N}$ such that $|P|=|\sigma(\overline{e})|$ and such that $\sigma(\overline{c})\perp P$, and let $\overline{e}'$ be the image of $\overline{e}$ under any automorphism of $\mathbb{N}$ taking $\sigma(\overline{e})$ to $P$. Now taking $\overline{v}=\overline{e}'$ gives a realization of $\phi(\overline{v})$ disjoint from $\overline{c}$, as desired.
\end{proof}

The following lemma is convenient when thinking about $\varnothing$-definable cofinal sets.

\begin{lemma}
Let $a\in\mathbb{M}$, and suppose there exist infinitely many $p\in\mathbb{P}$ such that $a_{p^\infty}\neq p^n$ for any $n\in\omega$. Then, for any $c\in\mathbb{M}$, there exists a conjugate $a'\equiv a$ such that $c\mid a'$.
\end{lemma}
\begin{proof}
After an automorphism, it suffices to find a realization of $\tp(c)$ that divides $a$. By compactness, we need only find a realization of a single formula $\theta(v)\in\tp(c)$ that divides $a$. Since $\mathbb{M}\models\exists v\theta(v)$, also $\mathbb{N}\models\exists v\theta(v)$, so let $q_1^{k_1}\cdot\ldots\cdot q_n^{k_n}$ be a realization of $\theta$ in $\mathbb{N}$, where the $q_i$ are distinct primes and the $k_i\in\omega$. By hypothesis, there exist primes $p_1,\dots,p_n\in\sigma(a)$ such that $a_{p_i^\infty}\neq p_i^k$ for any $k\in\omega$ and $i\in[n]$. In particular, each $p_i^{k_i}$ divides $a_{p_i^\infty}$, so $p_1^{k_1}\cdot\ldots\cdot p_n^{k_n}$ is a realization of $\theta$ dividing $a$, as desired.
\end{proof}

\subsection{Definability}

Given a tuple $\overline{a}\in\mathbb{M}^n$, denote its components by $a_1\dots a_n$, and define a partial ordering $\preccurlyeq$ on $n$-tuples by taking $\overline{a}^1\preccurlyeq\overline{a}^2$ to hold if and only if $a^1_m\mid a^2_m$, where $m\in[n+1]$ is minimal such that $a^1_m\neq a^2_m$. As usual, write $\overline{a}^1\prec\overline{a}^2$ to mean that $\overline{a}^1\preccurlyeq\overline{a}^2$ and $\overline{a}^1\neq\overline{a}^2$. This relation preserves several properties of divisibility – for example, to check that $\overline{a}^1\preccurlyeq\overline{a}^2$ it suffices to show that $\overline{a}^1_{p^\infty}\preccurlyeq\overline{a}^2_{p^\infty}$ for every $p\in\mathbb{P}$ – but it has the advantage that it descends to a total ordering on $(p^\infty)^n\subset\mathbb{M}^n$ even when $n>1$.

\begin{lemma}
Any non-empty $\mathbb{M}$-definable set $D\subseteq\mathbb{M}^n$ has $\preccurlyeq$-minimal elements: tuples $\overline{d}\in D$ such that, for any $\overline{e}\in D$, if $\overline{e}\preccurlyeq\overline{d}$ then $\overline{e}=\overline{d}$.
\end{lemma}
\begin{proof}
Let $\phi(\overline{v};\overline{w})$ be a formula and $\overline{c}\in\mathbb{M}^{|\overline{w}|}$ a tuple such that $D=\phi(\mathbb{M}^n;\overline{c})$. Since, by well-ordering of $\mathbb{N}$, \textit{every} non-empty subset of $\mathbb{N}^n$ has $\preccurlyeq$-minimal elements, we have the following theorem of $\mathbb{N}$, $$\mathbb{N}\models\forall\overline{w}\bigg[\exists\overline{v}\phi(\overline{v};\overline{w})\to\exists \overline{v}\left[\phi(\overline{v};\overline{w})\wedge\forall \overline{u}\big(\phi(\overline{u};\overline{w})\wedge (\overline{u}\preccurlyeq \overline{v})\to (\overline{u}=\overline{v})\big)\right]\bigg],$$ from which the desired result follows.
\end{proof}

The following two lemmas are very useful, giving a degree of control over the definable closure of a tuple $\overline{c}$; together they imply that any $\preccurlyeq$-minimal element of a $\overline{c}$-definable set $D$ is actually \textit{definable} over $\{\overline{c}\}\cup\sqrt{\mathbb{M}}$. Put differently, the family of $\preccurlyeq$-minimal elements of $D$ can be parametrized over $\overline{c}$ by a set of radical tuples. Refinements of this fact, in the form of Lemma 2.15 and Lemma 4.8, play a key role in the proofs of both of our main results.

\begin{lemma}
Suppose $D\subseteq\mathbb{M}^n$ is a $\overline{c}$-definable set, and that $\overline{d}\in D$ is $\preccurlyeq$-minimal in $D$. Then $\overline{d}_{p^\infty}\in\dcl\{\overline{c}_{p^\infty},p\}$ for each $p\in\mathbb{P}$.
\end{lemma}
\begin{proof}
Fix $p\in\mathbb{P}$ and suppose that $\overline{d}_{p^\infty}\notin\dcl\{\overline{c}_{p^\infty},p\}$. Then, by compactness, we may find some $\overline{q}\in(p^\infty)^n$ such that $\overline{q}\equiv_{\overline{c}_{p^\infty}}\overline{d}_{p^\infty}$ and $\overline{q}\prec\overline{d}_{p^\infty}$; otherwise $\overline{d}_{p^\infty}$ would be $\preccurlyeq$-minimal in $(p^\infty)^n$ satisfying some $\overline{c}_{p^\infty}$-formula, contradicting that $\overline{d}_{p^\infty}\notin\dcl\{\overline{c}_{p^\infty},p\}$. Now by Lemma 2.4 we have $\overline{d}_{\neg\{p\}}\cdot\overline{q}\equiv_{\overline{c}}\overline{d}$, whence $\overline{d}_{\neg\{p\}}\cdot\overline{q}\in D$, contradicting $\preccurlyeq$-minimality of $\overline{d}$.
\end{proof}

\begin{lemma}
Suppose $\overline{a},\overline{c}$ are tuples such that $\overline{a}_{p^\infty}\in\dcl\{\overline{c}_{p^\infty},p\}$ for each $p\in\mathbb{P}$. Then $\overline{a}$ is definable over $\{\overline{c}\}\cup\sqrt{\mathbb{M}}$.
\end{lemma}
\begin{proof}
By hypothesis, the following partial type in $u$ is inconsistent: $$\{\pi(u)\}\cup\left\{\phi(\overline{a}_{u^\infty},\overline{c}_{u^\infty})\to\exists^{>1}\overline{v}\left[\overline{v}\in(u^\infty)^n\wedge\phi(\overline{v},\overline{c}_{u^\infty})\right]\right\}_\phi.$$ Hence, by compactness, there exist finitely many formulas $\phi_1,\dots,\phi_k$ such that, for every prime $p\in\mathbb{P}$, one of $\phi_i((p^\infty)^n,\overline{c}_{p^\infty})=\{\overline{a}_{p^\infty}\}$ holds for some $i\in [k]$. Now, by Remark 2.3, for each $i\in[k]$ define $$r_i=\prod\big\{p\in\sigma(\overline{a}):\phi_i((p^\infty)^n,\overline{c}_{p^\infty})=\{\overline{a}_{p^\infty}\}\big\}\in\sqrt{\mathbb{M}}.$$ Then $\overline{a}$ is definable over $\{\overline{c},\overline{r}\}$.
\end{proof}

The next lemma says that a set definable over a tuple with support $P_1$ and definable over a tuple with support $P_2$ is definable over a tuple with support $P_1\cap P_2$. In Section 4 this allows us to show that every $\mathbb{M}$-definable set is definable over a smallest prime support.

\begin{lemma}
Let $\overline{c}^1,\overline{c}^2$ be any (finite) tuples, and let $P=\sigma(\overline{c}^1)\cap\sigma(\overline{c}^2)$. If $D\subseteq\mathbb{M}^n$ is $\overline{c}^1$-definable and $\overline{c}^2$-definable, then $D$ is $\overline{c}^1_P$-definable.
\end{lemma}
\begin{proof}
Let $f$ be any automorphism fixing $\overline{c}^1_P$; we will show that $f$ fixes $D$ setwise. First note that, since $f(\overline{c}^1_P)=\overline{c}^1_P$ and $P\subseteq\sigma(\overline{c}^1)$, we have $f(P)=P$ and thus $f(\overline{c}^1_{\neg P})\perp P$. In particular, defining $Q=\sigma(f(\overline{c}_{\neg P}^1))\cap\sigma(\overline{c}^2)$, we have $Q\perp P$ and $Q\subseteq\sigma(\overline{c}^2)$, and hence $Q\perp\overline{c}^1$. Thus $f(\overline{c}^1_{\neg P})_{Q}\perp f(\overline{c}_{\neg P}^1)_{\neg Q}\overline{c}^1_{\neg P}\overline{c}^1_P$, and so by Lemma 2.4 and Lemma 2.6 we may find an automorphism $g$ fixing each of the latter three tuples and such that $g(f(\overline{c}^1_{\neg P})_{Q})\perp\overline{c}^2$. Since $g$ fixes $\overline{c}^1$, it fixes $D$ setwise, and furthermore we have $$g(f(\overline{c}_{\neg P}^1))=g(f(\overline{c}_{\neg P}^1)_{Q})\cdot f(\overline{c}_{\neg P}^1)_{\neg Q},$$ whence $g(f(\overline{c}^1_{\neg P}))\perp\overline{c}^2$; in particular $g(f(\overline{c}^1_{\neg P}))\perp\overline{c}^1_P$ as well. Thus, since $g(f(\overline{c}^1_{\neg P}))\equiv\overline{c}^1_{\neg P}$ and $\overline{c}^1_{\neg P}\perp \overline{c}^1_{P}\overline{c}^2$, by Lemma 2.4 we may find an automorphism $h$ fixing these latter two tuples and taking $g(f(\overline{c}^1_{\neg P}))$ to $\overline{c}^1_{\neg P}$. Since $h$ fixes $\overline{c}^2$, it fixes $D$ setwise, so $h\circ g$ fixes $D$ setwise. On the other hand, note that each of $f,g,h$ fix $\overline{c}^1_P$, and that $h(g(f(\overline{c}^1_{\neg P})))=\overline{c}^1_{\neg P}$ by construction of $h$. In particular, $h\circ g\circ f$ fixes $\overline{c}^1$ and thus fixes $D$ setwise. Combining these two facts tells us that $f$ fixes $D$ setwise, as desired.
\end{proof}

Finally, we record a construction that generalizes the notion of a greatest common divisor. Since the restriction of $\preccurlyeq$ to $(p^\infty)^n\subset\mathbb{M}^n$ is a total ordering for every $p\in\mathbb{P}$, if $D\subseteq (p^\infty)^n$ in Lemma 2.8, then $D$ has a \textit{unique} $\preccurlyeq$-minimal element, which we denote by $\min_{\preccurlyeq}(D)$.

\begin{lemma}
Let $D\subseteq\mathbb{M}^n$ be a non-empty $\mathbb{M}$-definable set of $n$-tuples. Then there exists a unique $n$-tuple $\gamma(D)\in\mathbb{M}^n$ such that $\gamma(D)_{p^\infty}=\min_{\preccurlyeq}\left\{\overline{d}_{p^\infty}:\overline{d}\in D\right\}$ for each $p\in\mathbb{P}$. In particular, we have $\sigma(\gamma(D))=\bigcap_{\overline{d}\in D}\sigma(\overline{d})$.
\end{lemma}
\begin{proof}
This can be proved exactly as Lemma 2.8, since the desired condition is first-order expressible and possible for every subset of $\mathbb{N}^n$. For the second part, note that $\overline{1}\preccurlyeq\overline{c}$ for every $n$-tuple $\overline{c}$. In particular, given $p\in\mathbb{P}$, we have $\gamma(D)_{p^\infty}=\overline{1}$ if and only if there exists some $\overline{d}\in D$ with $\overline{d}_{p^\infty}=\overline{1}$.
\end{proof}

Note that $\gamma(D)\preccurlyeq\overline{d}$ for every $\overline{d}\in D$. If $n=1$, ie if $D\subseteq\mathbb{M}$, then the element $\gamma(D)\in\mathbb{M}$ is furthermore universal with that property: if $e\in\mathbb{M}$ is such that $e\preccurlyeq d$, ie $e\mid d$, for every $d\in D$, then also $e\preccurlyeq \gamma(D)$, ie $e\mid \gamma(D)$. In other words, in that case, $\gamma(D)$ is precisely the greatest common divisor of $D$.

\subsection{Cofinal sets}
In this section we will consider only $\varnothing$-definable sets of $1$-tuples. If $D\subseteq\mathbb{M}$ is such, then say that $D$ is \textit{cofinal} if it is cofinal with respect to divisibility in $\mathbb{M}$: in other words, if, for every $c\in\mathbb{M}$, there exists some $d\in D$ such that $c\mid d$. Fix a cofinal $\varnothing$-definable set $D\subseteq\mathbb{M}$.

\begin{lemma}
We have $\dcl D\supseteq\sqrt{\mathbb{M}}$.
\end{lemma}
\begin{proof}
Fix a radical element $r\in\sqrt{\mathbb{M}}$. By cofinality of $D$, there exists $d\in D$ such that $r\mid d$. By Lemma 2.6, we may find $e$ a conjugate of $d_{\neg\sigma(r)}$ such that $e\perp d$. Now by Lemma 2.4 we have $d=d_{\neg\sigma(r)}\cdot d_{\sigma(r)}\equiv e\cdot d_{\sigma(r)}$, so by $\varnothing$-definability $e\cdot d_{\sigma(r)}\in D$. But $r$ is now definable over $\{d,e\cdot d_{\sigma(r)}\}$, since its prime factors are precisely $\sigma(d)\cap\sigma(e\cdot d_{\sigma(r)})$.
\end{proof}

We require one fact that we will not prove here. First note that, for any $p\in\mathbb{P}$, the substructure $p^\infty$ of $\mathbb{M}$ is a model of Presburger arithmetic, since the $p$-adic valuation map induces a monoid isomorphism between $p^\infty\cap\mathbb{N}$ and $(\omega,+)$ whenever $p\in\mathbb{P}\cap\mathbb{N}$.

\begin{fact}
Let $p\in\mathbb{P}$ and suppose that $q_1,q_2\in p^\infty$ are powers of $p$. If $q_1\mid q_2$ and $q_2$ lies in $\dcl\{q_1,p\}$, then also $q_1\in\dcl\{q_2\}$.
\end{fact}
\begin{proof}
This is a case of `exchange' for Presburger arithmetic; see Example 10 in \citep{belegradek_peterzil_wagner_2000}.
\end{proof}

\begin{lemma}
We have $\dcl D=\mathbb{M}$.
\end{lemma}
\begin{proof}
Let $c\in\mathbb{M}$. By cofinality of $D$, the $c$-definable set $\{d\in D:c\mid d\}$ is non-empty, and thus by Lemma 2.8 it has a minimal element $d$. By Lemma 2.9, $d_{p^\infty}\in\dcl\{c_{p^\infty},p\}$ for each $p\in\mathbb{P}$. Since $c\mid d$, also $c_{p^\infty}\mid d_{p^\infty}$ for each $p\in\mathbb{P}$, and thus by Fact 2.14 $c_{p^\infty}\in\dcl\{d_{p^\infty}\}$ for each $p\in\mathbb{P}$. Now by Lemma 2.10, $c$ is definable over $\{d\}\cup\sqrt{\mathbb{M}}$, and since $\sqrt{\mathbb{M}}\subseteq\dcl D$ by Lemma 2.13, this means that $c$ is $D$-definable, as desired.
\end{proof}

\section{Stably embedded sets}

Recall that a $\varnothing$-definable set $D\subseteq\mathbb{M}^n$ is said to be `stably embedded' if every $\mathbb{M}$-definable subset of $D^k$ is definable with parameters from $D$. \citep[p.~136]{tent_ziegler_2012} In this section, we give a full characterization of the $\varnothing$-definable stably embedded sets of $T$. Note that, by Lemma 2.15, every cofinal $\varnothing$-definable subset of $\mathbb{M}$ is stably embedded.

\begin{lemma}
 The subset $\sqrt{\mathbb{M}}\subset\mathbb{M}$ is stably embedded.
\end{lemma}
\begin{proof}
Fix a formula $\phi(\overline{v};\overline{w})$ and a tuple $\overline{c}\in\mathbb{M}^{|\overline{w}|}$, and denote $n=|\overline{v}|$; we want to show that $\phi(\sqrt{\mathbb{M}}{}^n;\overline{c})$ is $\sqrt{\mathbb{M}}$-definable. Thus we will construct a formula $\psi(\overline{v};\overline{s})$, with parameters $\overline{s}$ from $\sqrt{\mathbb{M}}$, such that $\mathbb{M}\models\phi(\overline{r};\overline{c})$ if and only if $\mathbb{M}\models\psi(\overline{r};\overline{s})$ for any radical $n$-tuple $\overline{r}$.
 
 By Theorem A.4, we may assume that $\phi(\overline{v};\overline{w})$ is of the form $\theta^{\geqslant k}(\overline{v};\overline{w})$ for some formula $\theta(\overline{v};\overline{w})$ and some natural number $k\in\omega$. By Remark 2.1, $\theta^{\geqslant k}(\overline{r};\overline{c})$ holds if and only if there are at least $k$ elements $p\in\mathbb{P}$ for which $\theta^{p}(\overline{r}_{p^\infty};\overline{c})$ holds. On the other hand, if $\overline{r}\in\sqrt{\mathbb{M}}{}^n$ is radical, then for any prime $p$ the tuple $\overline{r}_{p^\infty}$ is of the form $p^{{\varepsilon}}:=p^{\varepsilon_1}\dots p^{\varepsilon_n}$ for some index tuple $\overline{\varepsilon}\in 2^n$. In light of this, for each index tuple $\overline{\varepsilon}\in 2^n$, we define a corresponding new formula $\theta_{{\varepsilon}}(u;\overline{w})\equiv\pi(u)\wedge\theta^u(u^{\varepsilon_1},\dots,u^{\varepsilon_n};\overline{w})$. Then, for any radical tuple $\overline{r}$, we have: $$\mathbb{M}\models\phi(\overline{r};\overline{c})\iff \mathbb{M}\models\theta_{{\varepsilon}}(p;\overline{c})\text{, where }\overline{r}_{p^\infty}=p^{{\varepsilon}},\text{ for at least }k\text{ elements }p\in\mathbb{P};$$ note that the formula $\overline{v}_{u^\infty}=u^{{\varepsilon}}$ is first-order expressible for any $\overline{\varepsilon}\in 2^n$.
 
 Now we are in a position to apply Lemma 2.2. For each $\overline{\varepsilon}\in 2^n$, there exists $\neg_{\varepsilon}$ either the empty string or the string `$\neg$' and a radical element $s_{\varepsilon}\in\sqrt{\mathbb{M}}$ whose prime support is precisely the set $\{p\in\mathbb{P}:\mathbb{M}\models\neg_{{\varepsilon}}\theta_{{\varepsilon}}(p;\overline{c})\}$. Now define
 $$\psi(\overline{v};\overline{s})\equiv\exists^{\geqslant k}u\left[\pi(u)\wedge\bigwedge_{{\varepsilon}\in 2^n}\left[\big(\overline{v}_{u^\infty}=u^{\varepsilon}\big)\to\neg_{\varepsilon}\big(u\mid s_{\varepsilon}\big)\right]\right].$$ By definition, the condition $u\mid s_\varepsilon$ is precisely the condition that $\mathbb{M}\models\neg_{{\varepsilon}}\theta_{{\varepsilon}}(u;\overline{c})$. Thus, for a radical tuple $\overline{r}\in\sqrt{\mathbb{M}}{}^n$, the sentence $\psi(\overline{r};\overline{s})$ says `there exist at least $k$ elements $p\in\mathbb{P}$ such that, if $\overline\varepsilon\in 2^n$ is such that $\overline{r}_{p^\infty}=p^{{\varepsilon}}$, then $\mathbb{M}\models\theta_{{\varepsilon}}(p;\overline{c})$'. Since every $\overline{r}_{p^\infty}$ will be of form $p^{{\varepsilon}}$ for some $\overline{\varepsilon}$, and the parameters $s_\varepsilon$ all lie in $\sqrt{\mathbb{M}}$, we are done.
\end{proof}

\begin{lemma}
If $D\subseteq\mathbb{M}$ is a $\varnothing$-definable and stably embedded set of $1$-tuples, and not a subset of $\{1\}$, then $\dcl D\supseteq\sqrt{\mathbb{M}}$.
\end{lemma}
\begin{proof}
 Suppose that some $r\in\mathbb{\sqrt{M}}$ is not $D$-definable. First note that $\sigma(r)$ is infinite, as we have the following: if $F\subseteq\mathbb{M}$ is \textit{any} $\varnothing$-definable set that is not a subset of $\{1\}$, then \textit{every} element decomposable as a product of finitely many primes is $F$-definable; to see this, it suffices to show that every $p\in\mathbb{P}$ is $F$-definable. Since $F$ is $\varnothing$-definable, we have $F\cap\N\nsubseteq\{1\}$, so choose any $a=q_1^{k_1}\cdot\ldots\cdot q_n^{k_n}\in F\cap\mathbb{N}\setminus\{1\}$ with the $q_i$ distinct primes and the $k_i>0$. We may find automorphisms $g,h$ of $\mathbb{M}$ such that $g(q_1)=h(q_1)=p$, and $g(q_i)\neq h(q_j)$ for any $i,j>1$. Then $p$ is $\{g(a),h(a)\}$-definable, since it is the only prime dividing both elements. On the other hand, $F$ is $\varnothing$-definable, so $g(a),h(a)\in F$, and thus we are done.
 
 So, $r$ has infinitely many prime factors. Now consider the subset $E\subseteq D$ defined by the formula $\sqrt{v}\mid{r}$. As above, choose any $a=q_1^{k_1}\cdot\ldots \cdot q_n^{k_n}\in D\cap\mathbb{N}\setminus\{1\}$, with the $q_i$ distinct primes and the $k_i>0$. Choose any distinct $p_1,\dots,p_n\in\sigma(r)$; then $a$ and $e:=p_1^{k_1}\cdot\ldots\cdot p_n^{k_n}$ are conjugate, and $\sqrt{e}\mid r$, so $e\in E$. Now we will show that $E$ is not $D$-definable; indeed, let $\overline{d}\in D^{|\overline{d}|}$ be a finite tuple. Since $r$ is not $D$-definable, there is an automorphism $f$ fixing $\overline{d}$ and not fixing $r$. Since a radical element is uniquely determined by its prime factors, we have two cases: either there is a prime $p\mid r$ such that $f(p)\nmid r$, or there is a prime $p\nmid r$ such that $f(p)\mid r$. By replacing $f$ with $f^{-1}$ if necessary, we may assume we are in the first case. Now we may find an automorphism $g$ that swaps $p_1$ and $p$ and fixes $\{p_2,\dots,p_n\}\setminus\{p\}$; in particular, since $p\mid r$, we still have $g(e)\in E$. On the other hand, since $p_1\mid e$, we have $p\mid g(e)$ and hence $f(p)\mid f(g(e))$. Since $f(p)\nmid r$, this means $f(g(e))\notin E$, and thus, since $g(e)\in E$, $f$ does not fix $E$ setwise. Since $\overline{d}$ was arbitrary, this show that $E$ is not $D$-definable, whence $D$ is not stably embedded. 
 \end{proof}

For terminological convenience, given a prime $p\in\mathbb{P}$ and an element $a\in\mathbb{M}$, say that $p$ is `non-standard in $a$' if $a_{p^\infty}\neq p^n$ for any $n\in\omega$. Say that $p$ is `standard in $a$' otherwise.

\begin{lemma}
If $D\subseteq\mathbb{M}$ is a $\varnothing$-definable and stably embedded set of $1$-tuples, and not a subset of $\dcl\sqrt{\mathbb{M}}$, then $\dcl D=\mathbb{M}$.
\end{lemma}
\begin{proof}
By Lemma 2.15, it suffices to show that $D$ is cofinal, so assume otherwise for contradiction. By Lemma 2.7, and since $D$ is definable without parameters, this means that there are only finitely many primes non-standard in $d$ for any $d\in D$. On the other hand, $D\nsubseteq\dcl\sqrt{\mathbb{M}}$, so let $d$ be an element of $D$ outside of $\dcl\sqrt{\mathbb{M}}$; we claim there is at least one prime non-standard in $d$. Otherwise, by compactness, there exists a finite bound $n_0\in\omega$ such that $d_{p^\infty}\mid p^{n_0}$ for every $p\in\mathbb{P}$. By Remark 2.3, there exists a radical element $s_i$ with support $\{p\in\sigma(d):d_{p^\infty}=p^i\}$ for each $i\in[n_0]$, and then $d$ is definable over $\overline{s}$, contradicting that $d\notin\dcl\sqrt{\mathbb{M}}$.

Now choose any element $c\in\mathbb{M}$ such that there are infinitely many primes non-standard in $c$. We claim that the subset $E\subseteq D$ defined by the formula $v\mid c$ is not $D$-definable, which will show the desired result. To see this, let $\overline{d}$ be any finite tuple from $D$. By construction of $c$, and since there are only finitely many primes non-standard in each $d_i$, there exists a prime $p\in\mathbb{P}$ non-standard in $c$ but standard in every $d_i$.

Now, by the first paragraph, let $d$ be any element of $D$ with at least one non-standard prime. By the proof of Lemma 2.7 and construction of $c$, there exists a conjugate $e\equiv d$ such that $e\mid c$ and such that $p$ is non-standard in $e$. Since $D$ is $\varnothing$-definable, $e\in D$, and hence $e\in E$. On the other hand, since $p$ is non-standard in $e$, we may find an element $q\in p^\infty$ such that $q\equiv e_{p^\infty}$ and such that $q\nmid c_{p^\infty}$. Since $p$ is standard in every $d_i$, we then have $q\equiv_{\overline{d}} e_{p^\infty}$, and so by Lemma 2.4 we also have $e_{\neg\{p\}}\cdot q\equiv_{\overline{d}} e$. But $q\nmid c_{p^\infty}$, so $e_{\neg\{p\}}\cdot q\notin E$, and hence $E$ is not $\overline{d}$-definable. Since $\overline{d}$ was arbitrary, $E$ is thus not $D$-definable, as desired.
\end{proof}

\begin{theorem}
A $\varnothing$-definable set $D\subseteq\mathbb{M}^n$ is stably embedded if and only if $\dcl D$ is one of the sets $\{1\}$, $\dcl\sqrt{\mathbb{M}}$, and $\mathbb{M}$.
\end{theorem}
\begin{proof}
The sets $\{1\}$ and $\mathbb{M}$ are trivially stably embedded, and $\sqrt{\mathbb{M}}$ is stably embedded by Lemma 3.1. Hence the backwards direction follows from Lemma C.2. For the forwards direction, again by Lemma C.2, we may assume that $n=1$, and then the desired implications are given in Lemma 3.2 and Lemma 3.3.
\end{proof}



\section{Imaginaries}
Let $D\subseteq\mathbb{M}^n$ be an $X$-definable set for some finite set $X\subset\mathbb{M}$. Recall that $X$ is a `canonical parameter' for $D$ if it is fixed pointwise by every automorphism of $\mathbb{M}$ fixing $D$ setwise, and that $X$ is an `almost canonical parameter' for $D$ if it is fixed setwise by every automorphism of $\mathbb{M}$ fixing $D$ setwise.

Further recall that a theory $T$ has `elimination of imaginaries' (EI) if every $\mathbb{M}$-definable set $D\subseteq\mathbb{M}^n$ has a canonical parameter, `elimination of finite imaginaries' if every finite subset of $\mathbb{M}^n$ has a canonical parameter, and `weak elimination of imaginaries' (WEI) if every $\mathbb{M}$-definable set $D\subseteq\mathbb{M}^n$ has an almost canonical parameter. \citep[p.~139]{tent_ziegler_2012} It is well-known that, to check EI and WEI, it suffices to assume that $D$ is an equivalence class of a $\varnothing$-definable equivalence relation $E\subseteq\mathbb{M}^m\times\mathbb{M}^m$; indeed, if $D=\phi(\mathbb{M}^n;\overline{c})$ for some formula $\phi(\overline{v};\overline{w})$ and some tuple $\overline{c}\in\mathbb{M}^{m}$, then we may define the equivalence relation $$E(\overline{w}_1,\overline{w}_2)\equiv\forall\overline{v}\left[\phi(\overline{v};\overline{w}_1)\leftrightarrow\phi(\overline{v};\overline{w}_2)\right],$$ and the equivalence class $\overline{c}^E$ will be fixed setwise by precisely the automorphisms of $\mathbb{M}$ fixing $D$ setwise.

\begin{remark} By Lemma 2.2, any $\mathbb{M}$-definable subset of $\mathbb{P}$ has a canonical parameter.
\end{remark}

We will show that Skolem arithmetic does not have elimination of finite imaginaries, but that it does have WEI; we prove the latter result by reducing the question from $\mathbb{M}$-definable sets to $\sqrt{\mathbb{M}}$-definable sets, the latter being more tractable.

\begin{theorem}
$T$ does not eliminate finite imaginaries.
\end{theorem}
\begin{proof}
Fix any distinct primes $p_1,p_2,q_1,q_2\in\mathbb{P}$, and define $a=p_1\cdot p_2$ and $b=q_1\cdot q_2$. We will show that $\{a,b\}$ has no canonical parameter. To see this, suppose that $e\in\mathbb{M}$ is any element fixed by all automorphisms fixing $\{a,b\}$ setwise. Then we claim that $$e_{p_1^\infty}\equiv e_{p_2^\infty}\equiv e_{q_1^\infty}\equiv e_{q_2^\infty}.$$ Indeed, the automorphisms of $\mathbb{M}$ fixing $\{a,b\}$ setwise are the disjoint union of those that fix both of $\{p_1,p_2\}$ and $\{q_1,q_2\}$ setwise and those that swap the sets $\{p_1,p_2\}$ and $\{q_1,q_2\}$. In particular, for each $r,s\in\{p_1,p_2,q_1,q_2\}$, there is an automorphism fixing $\{a,b\}$ setwise – hence fixing $e$ – and swapping $r$ and $s$, whence the claim follows. Thus if $X\subset\mathbb{M}$ is any finite subset fixed pointwise by every automorphism of $\mathbb{M}$ fixing $\{a,b\}$ setwise, then we may find an automorphism fixing $X$ pointwise, swapping $p_1$ and $q_1$, and fixing $p_2$ and $q_2$. This automorphism takes $a$ to $p_1\cdot q_1$ and $b$ to $p_2\cdot q_2$, hence does not fix $\{a,b\}$ setwise, and the result follows.
\end{proof}

Now we begin showing that $T$ has WEI; throughout the rest of this section, we will fix a $\varnothing$-definable equivalence relation $E\subseteq\mathbb{M}^n\times\mathbb{M}^n$ and a tuple $\overline{a}\in\mathbb{M}^n$; we will show that $\overline{a}^E$ has an almost canonical parameter. We begin by showing that $\overline{a}^E$ is definable over a smallest prime support. Recall the $n$-tuple $\gamma(\overline{a}^E)$ defined in Lemma 2.12, each $p$-part of which is $\min_{\preccurlyeq}\left\{\overline{a}'_{p^\infty}:\overline{a}'\in \overline{a}^E\right\}$, and recall in particular that $\sigma(\gamma(\overline{a}^E))=\bigcap_{\overline{a}'\in\overline{a}^E}\sigma(\overline{a}')$.

\begin{lemma}
There exists $\overline{a}'\in \overline{a}^E$ such that $\sigma(\overline{a})\cap\sigma(\overline{a}')=\sigma(\gamma(\overline{a}^E))$.
\end{lemma}
\begin{proof}
It suffices to show that $$\mathbb{N}\models\forall \overline{v}\exists \overline{v}'\left[E(\overline{v},\overline{v}')\wedge \left(\sigma(\overline{v})\cap\sigma(\overline{v}')=\sigma(\gamma(\overline{v}^E))\right)\right],$$ so we may assume that $\overline{a}\in\mathbb{N}^n$ and that $E\subseteq\mathbb{N}^n\times\mathbb{N}^n$. Let $\overline{a}'\in \overline{a}^E$ be such that $|\sigma(\overline{a})\cap\sigma(\overline{a}')|$ is minimal possible; we claim that $\overline{a}'$ has the desired property. If not, then there exists some $\overline{a}''\in \overline{a}^E$ such that $\sigma(\overline{a})\cap\sigma(\overline{a}')\nsubseteq\sigma(\overline{a}'')$; let $p\in\left(\sigma(\overline{a})\cap\sigma(\overline{a}')\right)\setminus\sigma(\overline{a}'')$, and let $q\in\mathbb{P}\cap\mathbb{N}$ be any prime not in $\sigma(\overline{a})\cup\sigma(\overline{a}')\cup\sigma(\overline{a}'')$. Now we may find an automorphism $f:\mathbb{N}\to\mathbb{N}$ swapping $p$ and $q$ and fixing all other primes in $\sigma(\overline{a})\cup\sigma(\overline{a}')\cup\sigma(\overline{a}'')$. Since $p\notin\sigma(\overline{a}'')$, $f$ fixes $\overline{a}''$, and hence fixes $\overline{a}^E$ setwise. In particular, $f(\overline{a}')\in a^E$. But also $$\sigma(\overline{a})\cap\sigma(f(\overline{a}'))=\sigma(\overline{a})\cap\sigma(\overline{a}')\setminus\{p\},$$ so this contradicts minimality of $|\sigma(\overline{a})\cap\sigma(\overline{a}')|$, as desired.
\end{proof}

\begin{corollary} $\overline{a}^E$ is definable over $\overline{a}_{\sigma(\gamma(\overline{a}^E))}$. Conversely, $\sigma(\gamma(\overline{a}^E))$ is the smallest prime support over which $\overline{a}^E$ is definable; in other words, if $\overline{a}^E$ is $\overline{c}$-definable for some tuple $\overline{c}$, then $\sigma(\gamma(\overline{a}^E))\subseteq\sigma(\overline{c})$.
\end{corollary}
\begin{proof}
First we show the first statement. By Lemma 4.3, we may find an element $\overline{a}'\in a^E$ such that $\sigma(\overline{a})\cap\sigma(\overline{a}')=\sigma(\gamma(\overline{a}^E))$. Since $\overline{a}^E$ is $\overline{a}'$-definable, by Lemma 2.11 $\overline{a}^E$ is also $\overline{a}_{\sigma(\overline{a}')}$-definable, and the desired result follows.

For the second statement, suppose $\overline{c}$ is otherwise, and let $p\in\sigma(\gamma(\overline{a}^E))\setminus\sigma(\overline{c})$. Now let $q$ be any prime outside of $\sigma(\overline{a})\cup\sigma(\overline{c})$. We may find an automorphism $f$ that fixes $\overline{c}$ and swaps $p$ and $q$. Since $f$ fixes $\overline{c}$, it must also fix $\overline{a}^E$ setwise. On the other hand, $p\notin \sigma(f(\overline{a}))$, so since $p\in\sigma(\gamma(\overline{a}^E))$ we have $f(\overline{a})\notin \overline{a}^E$, a contradiction.
\end{proof}

The following two lemmas show that every $\sqrt{\mathbb{M}}$-definable set has an almost canonical parameter; combined with stable embeddedness of $\sqrt{\mathbb{M}}$, this is enough to show that every $\mathbb{M}$-definable subset of $\sqrt{\mathbb{M}}{}^m$ has an almost canonical parameter.

\begin{lemma}
If $D\subseteq\mathbb{M}^n$ is $\sqrt{\mathbb{M}}$-definable, then there exists a tuple $\overline{r}\in\sqrt{\mathbb{M}}{}^m$ over which $D$ is definable, such that (i) $r_i\perp r_j$ for each $i\neq j$, and (ii) $\sigma(\overline{r})$ is the smallest prime support over which $D$ is definable.
\end{lemma}
\begin{proof}
Suppose $D$ is definable over a $k$-tuple $\overline{s}\in\sqrt{\mathbb{M}}{}^k$. By uniqueness of base $2$ representations, the tuple $\overline{s}$ is interdefinable with the element $$s:=s_1\cdot s_2^2\cdot\ldots\cdot s_k^{2^{k-1}}\in\mathbb{M},$$ so there exists a formula $\phi(\overline{v};w)$ with $D=\phi(\mathbb{M}^n;s)$. As usual, take $$E(w_1,w_2)\equiv \forall\overline{v}\left[\phi(\overline{v};w_1)\leftrightarrow\phi(\overline{v};w_2)\right].$$ By Corollary 4.4, $s^E$ and hence $D$ are definable over $s':=s_{\sigma(\gamma(s^E))}$. Since also $s'\mid s$, for every $p\in\sigma(s')$ we have $s'_{p^\infty}=p^i$ for some $i\in[2^{k}-1]$. Thus take $m=2^{k}-1$ and define a radical element $r_i=\prod\left\{p\in\sigma(s'):s'_{p^\infty}=p^i\right\}$ for each $i\in[m]$; then $s'$ is definable over $\overline{r}$, $r_i\perp r_j$ for each $i\neq j$, and $\sigma(\overline{r})=\sigma(s')=\sigma(\gamma(s^E))$ is the smallest support over which $D$ is definable, again by Corollary 4.4, so this gives the desired result.
\end{proof}

\begin{lemma}
If $D\subseteq\mathbb{M}^n$ is definable over some tuple $\overline{r}\in\sqrt{\mathbb{M}}{}^m$, where the $r_i$ satisfy the conditions of Lemma 4.5 and $m$ is minimal possible, then every automorphism fixing $D$ setwise permutes the $r_i$.
\end{lemma}
\begin{proof}
Let $D=\phi(\mathbb{M}^n;\overline{r})$ for some formula $\phi(\overline{v};\overline{w})$, and suppose for contradiction that $\alpha$ is an automorphism fixing $D$ setwise but not permuting the $r_i$. Since $\sigma(\overline{r})$ is the smallest support over which $D$ is definable, we have $\sigma(\alpha(\overline{r}))=\sigma(\overline{r})$. Since also $r_i\perp r_j$ for each $i\neq j$, there must hence be some $k\in[m]$ such that $\sigma(\alpha(r_k))\nsubseteq\sigma(r_i)$ for any $i\in[m]$. Thus $\sigma(\alpha(r_k))$ has non-empty intersection with at least two of the $\sigma(r_i)$; without loss of generality suppose there exists $q_l\in\sigma(\alpha(r_k))\cap\sigma(r_l)$ for each $l\in\{1,2\}$. We will show that $D$ is definable over $\{r_1\cdot r_2,r_3,\dots,r_m\}$, which will contradict minimality of $m$ and hence give the desired result. First we require the following:

\begin{claim}
For any finite subsets $P_1\subseteq\sigma(r_1)$ and $P_2\subseteq\sigma(r_2)$ with $|P_1|=|P_2|$, there is an automorphism fixing $D$ setwise, swapping $P_1$ and $P_2$, and fixing $\overline{r}_{\neg(P_1\cup P_2)}$.
\end{claim}
\begin{claimproof}
Let $i=|P_l|$ for $l\in\{1,2\}$; we prove the claim by induction on $i$. The case $i=0$ is trivial, since we may simply take the automorphism to be the identity. Hence suppose each $|P_l|=i+1$, so that $P_l=P_l'\sqcup\{p_l\}$ for some $P_l'\subset P_l$ of size $i$. Assume without loss of generality that $q_l\notin P'_l$ for either $l$, so that if $q_l\in P_l$ then $q_l=p_l$. Now, by the induction hypothesis we may find an automorphism $f$ that has the desired properties with respect to $P_1'$ and $P_2'$. Since $f$ fixes $\overline{r}_{\neg(P_1'\cup P_2')}$ and $p_l,q_l\notin P_1'\cup P_2'$ for either $l\in\{1,2\}$, we have $p_l,q_l\in\sigma(f(r_l))$ for each $l\in\{1,2\}$. Thus we may first find an automorphism $f'$ that fixes $f(\overline{r})$ and fixes $P_1'\cup P'_2$ pointwise, and swaps each $p_l$ with each $f(p_l)$, and then similarly find an automorphism $g$ that fixes $f(\overline{r})$ and fixes $P'_1\cup P'_2$ pointwise, and swaps each $p_l$ with each $q_l$. Since $f$ fixes $D$ setwise, $D$ is $f(\overline{r})$-definable, so both $f'$ and $g$ also fix $D$ setwise. On the other hand, $q_1,q_2\in\sigma(\alpha(r_k))$, so we may find an automorphism $h$ swapping $q_1$ and $q_2$, fixing $\alpha(\overline{r})$ and $f(\overline{r})_{\neg\{q_1,q_2\}}$, and fixing $P_1\cup P_2\setminus\{q_1,q_2\}$ pointwise; since $h$ fixes $\alpha(\overline{r})$, it also fixes $D$ setwise. Thus let $f''=g\circ h\circ g$; we claim that $f''\circ f'\circ f$ has the desired properties. Each of $f$, $f'$, and $f''$ fix $D$ setwise, so $f''\circ f'\circ f$ does too. Since $f''$ and $f'$ fix $P'_1\cup P'_2$ pointwise, and $f$ swaps $P_1'$ and $P_2'$, $f''\circ f'\circ f$ swaps $P'_1$ and $P_2'$ too. On the other hand, we have \begin{align*}(f''\circ f'\circ f)(p_l)&=(g\circ h\circ g)(f'(f(p_l)))\\&=(g\circ h\circ g)(p_l)\\&=(g\circ h)(q_{l})=g(q_{3-l})=p_{3-l}
\end{align*}for each $l\in\{1,2\}$, so that in fact $f''\circ f'\circ f$ swaps $P_1$ and $P_2$. For the final property, note that $f$ fixes $\overline{r}_{\neg(P_1'\cup P_2')}$, and that $f'$ fixes $f(\overline{r})$ and $P'_1\cup {P_2'}$, so that $f'$ fixes $\overline{r}_{\neg(P'_1\cup P'_2)}$. But $f'$ also swaps each $p_l$ and each $f(p_l)$, so in fact $f'\circ f$ fixes $\overline{r}_{\neg (P_1\cup P_2)}$. Similarly, $g\circ h\circ g$ fixes $f(\overline{r})_{\neg\{p_1,p_2\}}$ and $P'_1\cup P'_2$, and thus fixes $\overline{r}_{\neg(P_1\cup P_2)}$ too, as desired.
\end{claimproof}

A corollary is that, for any finite subsets $P_1,P_2\subseteq\sigma(r_1)\cup\sigma(r_2)$ with $|P_1|=|P_2|$, there is an automorphism fixing $D$ setwise, swapping $P_1$ and $P_2$, and fixing $\overline{r}_{\neg(P_1\cup P_2)}$. Now, consider the $\mathbb{M}$-definable set 
\begin{align*}
X:=\big\{r'_1r'_2\in\sqrt{\mathbb{M}}{}^2:\ &r'_1\cdot r'_2=r_1\cdot r_2,r'_1 r'_2\preccurlyeq r_1r_2,r'_1\perp r'_2,\\&\text{ and }D=\phi(\mathbb{M}^n;r'_1,r'_2,r_3,\dots,r_m)\big\}.
\end{align*} By Lemma 2.8, $X$ has a $\preccurlyeq$-minimal element $r'_1r'_2$. Since $r'_1r'_2\preccurlyeq r_1r_2$, we have $\sigma(r'_1)\subseteq\sigma(r_1)$; since also $r_1'\cdot r_2'=r_1\cdot r_2$, we have $\sigma(r'_2)\supseteq\sigma(r_2)$ as well.

\begin{claim}
One of $\sigma(r'_1)$ and $\sigma(r'_2)$ is finite.
\end{claim}
\begin{claimproof}
Suppose for contradiction that both sets are infinite. Then, by Claim 1 and compactness, we may find a $2$-tuple $r''_1r''_2\in\sqrt{\mathbb{M}}{}^2$ that is conjugate to $r_1r_2$ and such that: (i) $r''_1\cdot r''_2=r_1\cdot r_2$, (ii) $D=\phi(\mathbb{M}^n;r''_1,r''_2,r_3,\dots,r_m)$, and (iii) $P_l:=\sigma(r'_l)\cap\sigma(r''_1)$ is infinite for each $l\in\{1,2\}$. Now, choose any $p\in P_1$. By Lemma 2.4 and Lemma 2.5, we have $$r''_{1,P_1}r''_{1,P_2}\equiv r''_{1,P_1\setminus\{p\}}r''_{1,P_2\cup\{p\}},$$ and we may thus find an automorphism $f$ taking the first tuple to the second and fixing all of the following tuples: $r_1''r_2''$ and $r_3\dots r_m$ and $(r_1'r_2')_{\neg\sigma(r''_1)}$. Since $f$ fixes $r''_1r''_2r_3\dots r_m$, by condition (ii) $f$ fixes $D$ setwise, whence $D=\phi(\mathbb{M}^n;f(r'_1),f(r'_2),r_3,\dots,r_m)$. Also, $f$ fixes $r_1\cdot r_2=r''_1\cdot r''_2$, so $f(r'_1)\cdot f(r'_2)=r_1\cdot r_2$ as well. On the other hand, $f$ fixes $\sigma(r_1')\setminus\sigma(r''_1)$ setwise, and $f(P_1)=P_1\setminus\{p\}$. Thus $\sigma(f(r'_1))=\sigma(r'_1)\setminus\{p\}$, whence $f(r'_1)f(r'_2)\prec r'_1r'_2$, so $f(r'_1)f(r'_2)$ lies in $X$ and contradicts $\preccurlyeq$-minimality of $r'_1r'_2$ in $X$.
\end{claimproof}

Now we can show that $D$ is $\{r_1\cdot r_2,r_3,\dots,r_m\}$-definable; to see this, let $f$ be any automorphism fixing the latter set pointwise. By Claim 2, one of $\sigma(r'_1)$ and $\sigma(r'_2)$ is finite. If the latter is finite, then the subset $\sigma(r_2)\subseteq\sigma(r'_2)$ is also finite, and by the remark following Claim 1 we can find an automorphism $g$ that fixes the product $r_1\cdot r_2$ and the tuple $r_3\dots r_m$, fixes $D$ setwise, and takes $\sigma(f(r_2))$ to $\sigma(r_2)$. Then $g\circ f$ fixes $\overline{r}$, hence fixes $D$ setwise, and so $f$ fixes $D$ setwise as well. The case when $\sigma(r'_1)$ is finite is nearly identical; by the remark following Claim 1, we can find an automorphism $g$ that fixes the product $r_1\cdot r_2$ and the tuple $r_3\dots r_m$, fixes $D$ setwise, and takes $\sigma(f(r'_1))$ to $\sigma(r'_1)$. Then $g\circ f$ fixes $r'_1r'_2r_3\dots r_m$, hence fixes $D$ setwise by definition of $X$, and so $f$ again fixes $D$ setwise. In either case, we have shown that $D$ is $\{r_1\cdot r_2,r_3,\dots,r_m\}$-definable, which contradicts minimality of $m$ and thus gives the desired result.
\end{proof}

\begin{corollary}
Any $\mathbb{M}$-definable subset of $\sqrt{\mathbb{M}}{}^n$ has an almost canonical parameter.
\end{corollary}
\begin{proof}
Let $D\subseteq\sqrt{\mathbb{M}}{}^n$ be $\mathbb{M}$-definable. By Lemma 3.1, ie since $\sqrt{\mathbb{M}}$ is stably embedded, $D$ is $\sqrt{\mathbb{M}}$-definable. Thus, by Lemma 4.5 and Lemma 4.6, $D$ is definable over some $\overline{r}\in\sqrt{\mathbb{M}}{}^m$ such that every automorphism of $\mathbb{M}$ fixing $D$ setwise permutes the $r_i$. So $\{r_1,\dots,r_m\}$ is an almost canonical parameter for $D$, and the result follows.
\end{proof}

Now return to the case where $E\subseteq\mathbb{M}^n\times\mathbb{M}^n$ is an arbitrary $\varnothing$-definable equivalence relation of $n$-tuples and $\overline{a}\in \mathbb{M}^n$ is an $n$-tuple. Equipped with Corollary 4.7, in order to show that $\overline{a}^E$ has an almost canonical parameter, we need only reduce the question to showing that a suitable definable subset of $\sqrt{\mathbb{M}}{}^m$ has an almost canonical parameter. We do this via the lemmas on definability proven in Section 2.4.

\begin{lemma}
$\overline{a}^E$ is definable over $\left\{\gamma(\overline{a}^E)\right\}\cup\sqrt{\mathbb{M}}$.
\end{lemma}
\begin{proof}
By Lemma 2.8, let $\overline{a}'$ be any $\preccurlyeq$-minimal element of $\overline{a}^E$. Since $\overline{a}^E$ is $\overline{a}'$-definable, it suffices to show that $\overline{a}'$ is definable over $\left\{\gamma(\overline{a}^E)\right\}\cup\sqrt{\mathbb{M}}$. Furthermore, by Lemma 2.10, it suffices for this to show that $\overline{a}'_{p^\infty}\in\dcl\left\{\gamma(\overline{a}^E)_{p^\infty},p\right\}$ for each $p\in\mathbb{P}$. Thus fix any $p\in\mathbb{P}$. By definition, $\gamma(\overline{a}^E)_{p^\infty}$ is $\preccurlyeq$-minimal among the $p$-parts of the elements of $\overline{a}^E$, so there exists some $\overline{a}''\in\overline{a}^E$ such that $\gamma(\overline{a}^E)_{p^\infty}=\overline{a}''_{p^\infty}$. Now $\overline{a}^E$ is $\overline{a}''$-definable; since $\overline{a}'$ is $\preccurlyeq$-minimal in $\overline{a}^E$, by Lemma 2.9 we thus have $\overline{a}'_{p^\infty}\in\dcl\left\{\overline{a}''_{p^\infty},p\right\}$, and the result follows.
\end{proof}

\begin{theorem}
$T$ has weak elimination of imaginaries.
\end{theorem}
\begin{proof}
We need to show that $\overline{a}^E$ has an almost canonical parameter. By Lemma 4.8, let $\overline{s}\in\sqrt{\mathbb{M}}{}^m$ be any radical tuple such that $\overline{a}^E$ is definable over $\left\{\gamma(\overline{a}^E),\overline{s}\right\}$, and choose a formula $\phi(\overline{v};\overline{u},\overline{w})$ such that $\overline{a}^E=\phi(\mathbb{M}^n;\gamma(\overline{a}^E),\overline{s})$. By Corollary 4.7, the $\mathbb{M}$-definable set $$D:=\left\{\overline{s}'\in\sqrt{\mathbb{M}}{}^m:\overline{a}^E=\phi(\mathbb{M}^n;\gamma(\overline{a}^E),\overline{s}')\right\}\subseteq\sqrt{\mathbb{M}}{}^m$$ has an almost canonical parameter $X\subset\mathbb{M}$. $D$ is definable over $X$, so $\overline{a}^E$ is definable over the union $\left\{\gamma(\overline{a}^E)\right\}\cup X$. On the other hand, every automorphism of $\mathbb{M}$ that fixes $\overline{a}^E$ setwise fixes $\gamma(\overline{a}^E)$, hence also fixes $D$ setwise, and hence fixes $X$ setwise as well. So each such automorphism fixes $\left\{\gamma(\overline{a}^E)\right\}\cup X$ setwise, and this latter set is thus an almost canonical parameter for $\overline{a}^E$.
\end{proof}

\newpage
\appendix
\section{Proof of QE}
\begin{remark}
Following the notation of Jeřábek \citep{jerabek_2018}, we will occasionally find it convenient to write the formula $\phi^{\geqslant n}(\overline{v})$ as $|S(\overline{v})|\geqslant n$, where $S(\overline{v})$ is the `set', which depends on $\overline{v}$, of the form $\left\{p\in\mathbb{P}:\phi^p(\overline{v})\right\}$. In particular, if $(S_i(\overline{v}))_{i\in I}$ is a family of such `sets', then any Boolean combination of (finitely many) of the $S_i(\overline{v})$ will also be of that form.
\end{remark}We will prove the desired result in a series of lemmas. Use $\neg_0$ to denote the string `$\neg$' and $\neg_1$ to denote the empty string. Also, given a natural number $k\in\omega$, let $2^k$ denote the set of $k$-tuples whose entries are all $0$ or $1$. We will use $2^k_0$ to mean $2^k\setminus\{\overline{0}\}$.

\begin{lemma}
For any formulas $\phi_i(\overline{v})$ and $\psi_j(\overline{v})$ and any natural numbers $m_i,n_j\in\omega$, the formula $$\eta(\overline{v})\equiv\bigwedge_{i\in[k]}\phi_i^{\leqslant m_i}(\overline{v})\wedge\bigwedge_{j\in[l]}\psi_j^{\geqslant n_j}(\overline{v})$$ is equivalent to a disjunction of formulas of the form $$\bigwedge_{\alpha\in I}\theta_\alpha^{=m'_\alpha}(\overline{v})\wedge\bigwedge_{\beta\in J}\xi_\beta^{\geqslant n_\beta'}(\overline{v})$$ such that, at any prime $p$, the family of formulas $\left\{\theta_\alpha^p(\overline{v}),\xi_\beta^p(\overline{v}):\alpha\in I,\beta\in J\right\}$ is pairwise inconsistent.
\end{lemma}
\begin{proof}
First note that $\eta(\overline{v})$ is equivalent to the disjunction $$\bigvee_{m'_1\leqslant m_1}\dots\bigvee_{m'_k\leqslant m_k}\left[\bigwedge_{i\in[k]}\phi_i^{=m'_i}(\overline{v})\wedge\bigwedge_{j\in[l]}\psi_j^{\geqslant n_j}(\overline{v})\right],$$ and thus we may assume that $\eta(\overline{v})$ is of form $\bigwedge_{i\in[k]}\phi_i^{=m_i}(\overline{v})\wedge\bigwedge_{j\in[l]}\psi_j^{\geqslant n_j}(\overline{v})$. Now, for any elements $\overline{\varepsilon}\in 2^k_0$ and $\overline{\varsigma}\in 2^l$, define two kinds of formula:
\begin{enumerate}
    \item Let $\theta_{\overline{\varepsilon}\overline{\varsigma}}(\overline{v})\equiv\bigwedge_{i\in[k]}\neg_{\varepsilon_i}\phi_i(\overline{v})\wedge\bigwedge_{j\in[l]}\neg_{\varsigma_j}\psi_j(\overline{v})$.
    \item Let $\xi_{\overline{\varsigma}}(\overline{v})\equiv\bigwedge_{i\in[k]}\neg\phi_i(\overline{v})\wedge\bigwedge_{j\in[l]}\neg_{\varsigma_j}\psi_j(\overline{v})$.
\end{enumerate}
Certainly the family $\left\{\theta_{\overline{\varepsilon\varsigma}}^p(\overline{v}):\overline{\varepsilon}\in 2^k_0,\overline{\varsigma}\in 2^l\right\}\cup\left\{\xi_{\overline{\varsigma}}^p(\overline{v}):\overline{\varsigma}\in 2^l_0\right\}$ is pairwise inconsistent for any prime $p\in\mathbb{P}$. With this in mind, let $F$ be the set of functions $f:2^{k+l}_0\to\omega$ with the following two properties:
\begin{enumerate}
    \item For each $i\in[k]$, we have $m_i=\sum_{\overline{\varepsilon\varsigma}\in 2^{k+l}_0}\varepsilon_if(\overline{\varepsilon\varsigma})$.
    \item For each $j\in[l]$, we have $n_j=\sum_{\overline{\varepsilon\varsigma}\in 2^{k+l}_0}\varsigma_jf(\overline{\varepsilon\varsigma})$.
\end{enumerate} Note that $F$ is finite, since the image of each $f\in F$ is bounded on each element of $2^{k+l}_0$. Indeed, if $\overline{\varepsilon\varsigma}\in 2_0^{k+l}$, there is either some $i\in[k]$ such that $\varepsilon_i=1$ or some $j\in[l]$ such that $\varsigma_j=1$. In the first case we have $f(\overline{\varepsilon\varsigma})\leqslant m_i$, and in the second case we have $f(\overline{\varepsilon\varsigma})\leqslant n_j$. In particular, the formula $$\bigvee_{f\in F}\left[\bigwedge_{\overline{\varepsilon}\in 2^k_0,{\overline{\varsigma}}\in 2^{l}}\theta_{\overline{\varepsilon\varsigma}}^{=f(\overline{\varepsilon\varsigma})}(\overline{v})\wedge\bigwedge_{\overline{\varsigma}\in 2^l_0}\xi_{\overline{\varsigma}}^{\geqslant f(\overline{0\varsigma})}(\overline{v})\right]$$ is well-defined and equivalent to $\eta(\overline{v})$.
\end{proof}

\begin{corollary}
For any formulas $\phi_i(\overline{v},w)$ and $\psi_j(\overline{v},w)$ and any $m_i,n_j\in\omega$, the formula $$\eta(\overline{v})\equiv\exists w\left[\bigwedge_{i\in[k]}\phi_i^{\leqslant m_i}(\overline{v},w)\wedge\bigwedge_{j\in[l]}\psi_j^{\geqslant n_j}(\overline{v},w)\right]$$ is equivalent to a Boolean combination of formulas of the form $\theta^{\geqslant m}(\overline{v})$.
\end{corollary}
\begin{proof}
By Lemma A.2, and since existential quantification distributes over disjunction, we may assume that $\eta(\overline{v})\equiv\exists w\left[\bigwedge_{i\in[k]}\phi_i^{=m_i}(\overline{v},w)\wedge\bigwedge_{j\in[l]}\psi_j^{\geqslant n_j}(\overline{v},w)\right]$ and that, for any $p\in\mathbb{P}$, the family $\left\{\phi_i^p(\overline{v},w),\psi_j^p(\overline{v},w):i\in[k],j\in[l]\right\}$ is pairwise inconsistent. (For convenience, call this latter fact $\ast$.) Consider the following sets of primes:
\begin{enumerate}
    \item $S_i(\overline{v}):=\big\{p\in\mathbb{P}:\exists w\phi_i^p(\overline{v},w)\big\}$, for each $i\in[k]$.
    \item $T_j(\overline{v}):=\big\{p\in\mathbb{P}:\exists w\psi_j^p(\overline{v},w)\big\}$, for each $j\in[l]$.
    \item $U(\overline{v}):=\left\{p\in\mathbb{P}:\forall w\bigvee_{i=1}^k\phi_i^p(\overline{v},w)\right\}$.
    \item $V_i:=\left\{p\in\mathbb{P}:\phi_i^p(\overline{1},1)\right\}$, for each $i\in[k]$.
\end{enumerate} We claim that, for any tuple $\overline{a}\in\mathbb{M}^{|\overline{a}|}$, the sentence $\eta(\overline{a})$ is equivalent to the following conditions on these sets:
\begin{enumerate}
    \item For each $i\in[k]$ and $j\in[l]$, there are subsets $P_i\subseteq S_i(\overline{a})$ and $Q_j\subseteq T_j(\overline{a})$ satisfying
    \begin{enumerate}
        \item $|P_i|=m_i$ and $|Q_j|=n_j$ for each $i\in[k]$, $j\in[l]$.
        \item $P_{i_1}\cap P_{i_2}=P_{i_1}\cap Q_{j_1}=Q_{j_1}\cap Q_{j_2}=\varnothing$ for any $i_1\neq i_2 \in[k]$, $j_1\neq j_2\in [l]$.
        \item $U(\overline{a})\subseteq\bigcup_{i\in[k]}P_i$.
    \end{enumerate}
    \item $|V_i|=0$ for each $i\in[k]$.
\end{enumerate} By Remark A.1 and Fact B.2, condition ($1$) is equivalent to a Boolean combination of formulas of the desired form. Clearly condition ($2$) is also of the desired form, and so this will prove the result.

So, we need only show the equivalence of these conditions. Note that we may assume $\overline{a}\in\mathbb{N}^{|\overline{a}|}$ and that all the sets above quantify over $\mathbb{P}\cap\mathbb{N}$ rather than $\mathbb{P}$. First we show the forward direction. Suppose $\mathbb{N}\models\eta(\overline{a})$ and that $c\in\mathbb{N}$ is a witness. We will show condition ($1$) first. For each $i\in[k]$, let $P_i=\{p\in\mathbb{P}\cap\mathbb{N}:\phi_i^p(\overline{a},c)\}$. Note that $P_i\subseteq S_i(\overline{a})$, with $c$ as a witness of the defining formula for $S_i(\overline{v})$, and that $|P_i|=m_i$ by definition of $\eta$. Similarly, by definition of $\eta$, for each $j\in[l]$ we may find a subset $Q_j\subseteq\{q\in\mathbb{P}\cap\mathbb{N}:\psi_j^q(\overline{a},c)\}$ with $|Q_j|=n_j$. Again, $Q_j\subseteq T_j(\overline{a})$, with $c$ as a witness. Thus the $P_i$ and $Q_j$ satisfy condition ($1$a). That they also satisfy condition ($1$b) is immediate from fact $\ast$. To see ($1$c), suppose $p\in U(\overline{a})$; then, instantiating $w$ at $c$, we have $\mathbb{N}\models\bigvee_{i\in[k]}\phi_i^p(\overline{a},c)$. By definition of the $P_i$, this means $p$ lies in $P_i$ for some $i\in[k]$, as desired. So, all of the properties needed for condition ($1$) are satisfied. To see that condition ($2$) holds, fix $i\in[k]$, and let $q\in\mathbb{P}\cap\mathbb{N}$ be any prime outside $P_i\cup\sigma(c)\cup\sigma(\overline{a})$. Since $q\notin P_i$, by definition we have $\mathbb{N}\models\neg\phi_i^q(\overline{a},c)$. By Remark 2.1, this means in particular that $\mathbb{N}\models\neg\phi_i^q(\overline{a}_{q^\infty},c_{q^\infty})$. But $\overline{a}_{q^\infty}=\overline{1}$ and $c_{q^\infty}=1$ by construction of $q$, so we have $\mathbb{N}\models\neg\phi_i^q(\overline{1},1)$. Now for any $p\in\mathbb{P}\cap\mathbb{N}$ there is an automorphism of $\mathbb{N}$ taking $q$ to $p$, and so we also have $\mathbb{N}\models\neg\phi_i^p(\overline{1},1)$. In particular, $|V_i|=0$.

Now we show the backward direction. Suppose that conditions ($1$) and ($2$) hold, witnessed by the sets $P_i$ and $Q_j$. We will create a witness $c$ for $\eta(\overline{a})$ by stitching together its $p$-parts in the following steps:
\begin{enumerate}
    \item Fix $i\in[k]$. For any $p\in P_i$, since $p\in S_i(\overline{a})$, there exists $e_p\in\mathbb{N}$ such that $\mathbb{N}\models\phi_i^p(\overline{a},e_p)$. By Remark 2.1, we may replace $e_p$ with $e_{p,p^\infty}$ and thus assume that is it a power of $p$. Then let $c_{p^\infty}=e_p$.
    \item Similarly, for every $j\in[l]$ and $q\in Q_j$, there is $f_q$ a power of $q$ such that $\mathbb{N}\models\psi_j^q(\overline{a},f_q)$. Let $c_{q^\infty}=f_q$.
    \item Now suppose $p\in\sigma(\overline{a})$. If $p\in U(\overline{a})$, then, by condition ($1$c), $p\in P_i$ for some $i\in[k]$, and so has been handled in step $1$. If $p\notin U(\overline{a})$, then by the defining formula of $U(\overline{v})$ there must exist some $g_p\in\mathbb{N}$ such that $\mathbb{N}\models\bigwedge_{i=1}^k\neg\phi_i^p(\overline{a},g_p)$. Again by Remark 2.1, we may assume that $g_p$ is a power of $p$, and then let $c_{p^\infty}=g_p$.
    \item Now suppose $p\notin\sigma(\overline{a})$ and that $p\notin\bigcup_{i\in[k]}P_i\cup\bigcup_{j\in[l]} Q_j$. Then let $c_{p^\infty}=1$. Note that, by condition ($2$), we have $\mathbb{N}\models\neg\phi_i^p(\overline{1},1)$ for any $i\in[k]$. In particular, since $\overline{a}_{p^\infty}=\overline{1}$, by Remark 2.1 we will have $\mathbb{N}\models\neg\phi_i^p(\overline{a},c)$.
\end{enumerate} This is well defined for two reasons. First, all but finitely many of the $p$-parts of $c$ will be $1$: the primes at which the $p$-part of $c$ are not $1$ lie either in $\sigma(\overline{a})$, which is finite, or in $\bigcup_{i\in[k]} P_i\cup\bigcup_{j\in[l]} Q_j$, which is finite by condition ($1$a). Second, the $P_i$ and the $Q_j$ are all pairwise disjoint by condition ($1$b), and so the primes specified at each stage in the above steps are all distinct. We claim now that $$\mathbb{N}\models\bigwedge_{i\in[k]}\phi_i^{=m_i}(\overline{a},c)\wedge\bigwedge_{j\in[l]}\psi_j^{\geqslant n_j}(\overline{a},c),$$ which will complete the proof. We prove this in two parts.
\begin{enumerate}
    \item Let $j\in[l]$. We claim that $Q_j\subseteq\{q\in\mathbb{P}\cap\mathbb{N}:\psi_j^q(\overline{a},c)\}$. Since $|Q_j|=n_j$ by condition ($1$a), this will show $\mathbb{N}\models\psi_j^{\geqslant n_j}(\overline{a},c)$, as desired. To see this, suppose $q\in Q_j$. Then by step $2$ we have $\mathbb{N}\models\psi_j^q(\overline{a},c_{q^\infty})$ and hence $\mathbb{N}\models\psi_j^q(\overline{a},c)$, as desired.
    \item Let $i\in[k]$. We claim that $P_i=\{p\in\mathbb{P}\cap\mathbb{N}:\phi_i^p(\overline{a},c)\}$. Since $|P_i|=m_i$ by condition ($1$a), this will show $\mathbb{N}\models\phi_i^{=m_i}(\overline{a},c)$, as desired. The inclusion $\subseteq$ follows from an identical argument as in the above item. So, we show $\supseteq$. Suppose therefore that some $p\in\mathbb{P}\cap\mathbb{N}$ does not lie in $P_i$. We have four cases; if $p\in P_{i'}$ for some $i'\neq i\in[k]$, then by the inclusion $\subseteq$ we have $\mathbb{N}\models\phi_{i'}^p(\overline{a},c)$. Since $\phi_{i'}^p(\overline{v},w)$ and $\phi_i^p(\overline{v},w)$ are inconsistent by $\ast$, this means $\mathbb{N}\models\neg\phi_i^p(\overline{a},c)$, as desired. The case when $p\in Q_j$ for some $j\in[l]$ follows similarly, since $\psi_j^p(\overline{v},w)$ and $\phi_i^p(\overline{v},w)$ are again inconsistent by $\ast$. If $p$ does not lie in any of the $P_{i'}$ or $Q_j$, but it is a prime factor of some element of $\overline{a}$, then by step $3$ we have $c_{p^\infty}=g_p$ and hence $\mathbb{N}\models\neg\phi_i^p(\overline{a},c_{p^\infty})$, as desired. Finally, if $p$ does not lie in any of the $P_{i'}$ or $Q_j$, and does not divide any element of $\overline{a}$, then by the remark in step $4$ we have $\mathbb{N}\models\neg\phi_i^p(\overline{a},c)$.
\end{enumerate} This concludes the proof.
\end{proof}

\begin{theorem}
Any formula $\phi(\overline{v})$ is equivalent to a Boolean combination of formulas of the form $\theta^{\geqslant m}(\overline{v})$.
\end{theorem}
\begin{proof}
Proof by induction on the complexity of $\phi(\overline{v})$. First suppose that $\phi$ is atomic; then it is of the form $v_1^{k_1}\dots v_n^{k_n}=v_1^{l_1}\dots v_n^{l_n}$ for some $k_i,l_i\in\omega$. Since two elements of $\mathbb{N}$ are equal if and only if their $p$-parts are equal for all $p\in\mathbb{P}$, and the $p$-part operation commutes with products, the formula $\phi(\overline{v})$ is equivalent to $(\neg\phi)^{\leqslant 0}(\overline{v})$, as desired. Now it suffices to show that, if $\phi(\overline{v},w)$ is a conjunction of formulas of the form $\theta^{\geqslant m}(\overline{v},w)$ and $\theta^{\leqslant m}(\overline{v},w)$, then so is $\exists w\phi(\overline{v},w)$. This follows from Corollary A.3.
\end{proof}

\section{A combinatorial fact}
Thank you to Alex Kruckman for his observation that Fact B.2 follows from the following theorem, and to Jamshid Derakhshan and Angus Macintyre for their elegant exposition of its proof in Example 1 of the paper \citep{derakhshan_macintyre_2017}.

\begin{theorem}
The theory of infinite atomic Boolean algebras has quantifier elimination after expanding its language to include the predicate `$v$ has $n$ atoms lying beneath it' for every $n\in\omega$.$\hfill\blacksquare$
\end{theorem}

\begin{fact}
Suppose we have natural numbers $m_i,n_j\in\omega$ and sets $S_i,T_j$, for each $i\in[k]$ and $j\in[l]$, and another set $U$. Then the existence of subsets $P_i\subseteq S_i$ and $Q_j\subseteq T_j$ such that
\begin{enumerate}
    \item $|P_i|=m_i$ and $|Q_j|=n_j$ for each $i\in[k]$, $j\in[l]$.
    \item $P_{i_1}\cap P_{i_2}=P_{i_1}\cap Q_{j_1}=Q_{j_1}\cap Q_{j_2}=\varnothing$ for any $i_1\neq i_2\in[k]$, $j_1\neq j_2\in[l]$.
    \item $U\subseteq\bigcup_{i=1}^k P_i$.
\end{enumerate} is expressible as a Boolean combination of formulas of the form $|X|\geqslant m$, where $m\in\omega$ and $X$ is a Boolean combination of the $S_i$, the $T_j$, and $U$.$\hfill\blacksquare$
\end{fact}

\section{Definable closures of stably embedded sets}

Let $T$ be any complete first-order theory, and let $\mathbb{M}\models T$ be a monster model in the sense of Section 0. Given a tuple $\overline{w}$ and a formula $\psi(v)$ in a single variable, use $\psi(\overline{w})$ to abbreviate the formula $\bigwedge_{i\in[|\overline{w}|]}\psi(w_i)$.

\begin{lemma}
Suppose $D\subseteq\mathbb{M}^m$ and $E\subseteq\mathbb{M}^n$ are $\varnothing$-definable, and that $D\subseteq \dcl E$. Then there is a $\varnothing$-definable surjection from a $\varnothing$-definable set of tuples of $E$ onto $D$.
\end{lemma}
\begin{proof}
For notational convenience, we assume $m=n=1$; the more general case is identical. Let $\phi(v)$ and $\psi(v)$ be such that $D=\phi(\mathbb{M})$ and $E=\psi(\mathbb{M})$. By hypothesis, the following partial type in $v$ is inconsistent:
$$\{\phi(v)\}\cup\left\{\forall\overline{w}\left[\psi(\overline{w})\wedge\theta(v;\overline{w})\to\exists^{>1}u\theta(u;\overline{w})\right]\right\}_\theta.$$ Thus, by compactness, there exist formulas $\theta_1(v;\overline{w}_1),\dots,\theta_k(v;\overline{w}_k)$ such that, for every element $d\in D$, there exist some $i\in[k]$ and a tuple $\overline{e}\in E^{|\overline{w}_i|}$ with $\theta_i(D;\overline{e})=\{d\}$. Without loss of generality, assume that the $\overline{w}_i$ are all tuples of the same length. For each $i\in[k]$, define a formula $$\delta_i(\overline{w})\equiv\psi(\overline{w})\wedge\exists^{=1}v\theta_i(v;\overline{w})\wedge\forall v\left[\theta_i(v;\overline{w})\to\phi(v)\right].$$ Now, defining the formula $\delta(\overline{w})\equiv\bigvee_{i\leqslant k}\delta_i(\overline{w})$, we get a $\varnothing$-definable surjection $g(\overline{w})$ from $\delta(E^{|\overline{w}|})$ onto $D$ by taking $g(\overline{w})$ to act as $\theta_i$ on each $\delta_i(E^{|\overline{w}|})\setminus\bigcup_{j<i}\delta_j(E^{|\overline{w}|})$.
\end{proof}

We need the following lemma in Section 3:
\begin{lemma}
Suppose $D\subseteq\mathbb{M}^m$ and $E\subseteq\mathbb{M}^n$ are $\varnothing$-definable, with $\dcl D=\dcl E$. Then $D$ is stably embedded if and only if $E$ is stably embedded.
\end{lemma}
\begin{proof}
Again suppose $m=n=1$ for convenience, and assume that $E$ is stably embedded. By Lemma C.1, there is a formula $\delta(\overline{w})$ and a $\varnothing$-definable surjection $g(\overline{w})$ from $\delta(E^{|\overline{w}|})$ onto $D$. Let $\xi(\overline{v};\overline{c})$ be any formula. Since $E$ is stably embedded, the set $$X:=\left\{\overline{e}_1\dots\overline{e}_{k}\in\delta(E^{|\overline{w}|})^k:\mathbb{M}\models \xi(g(\overline{e}_1),\dots,g(\overline{e}_{k});\overline{c})\right\}$$ is $E$-definable. Since $g$ is surjective, the image of $X$ under $g$ is precisely $\xi(D^k;\overline{c})$, and so, since $g$ is $\varnothing$-definable, $\xi(D^k;\overline{c})$ is also $E$-definable. Since $E\subseteq\dcl D$, we are done.
\end{proof}


\end{document}